\documentclass[a4paper]{article}
\newif\ifnlaa
\newif\ifpgfgen

\nlaafalse 
\pgfgenfalse 

\usepackage[l2tabu, orthodox]{nag}
\usepackage{amsmath,amsfonts,amssymb,amsthm}
\usepackage{hyperref}
\usepackage{algorithm}
\usepackage[noend]{algpseudocode}
\usepackage[group-separator={\ }]{siunitx}
\usepackage{booktabs,colortbl,multirow}
\usepackage{todonotes}
\usepackage{fullpage}
\usepackage{microtype}
\usepackage{subcaption}

\ifnlaa
\usepackage[super,numbers,sort&compress]{natbib} 
\else
\usepackage[numbers,sort&compress]{natbib}
\fi

\ifpgfgen
\usepackage{tikz}
\usepackage{pgfplots}
\pgfplotsset{compat=1.13}
\usetikzlibrary{external}
\usetikzlibrary{pgfplots.statistics}
\tikzsetexternalprefix{figures/}
\usepgfplotslibrary{statistics}
\makeatletter
\renewcommand{\todo}[2][]{\tikzexternaldisable\@todo[#1]{#2}\tikzexternalenable}
\makeatother
\else 
\usepackage{graphicx}
\usepackage{tikzexternal}
\fi

\usepackage{tikzscale}
\graphicspath{{./img/}{./figures/}{.}}

\newtheorem{theorem}{Theorem}
\newtheorem{lemma}{Lemma}
\theoremstyle{definition}
\newtheorem*{remark}{Remark}

\reversemarginpar

\usepackage[noabbrev]{cleveref}
\makeatletter
\newcounter{algorithmicH}
\let\oldalgorithmic\algorithmic
\renewcommand{\algorithmic}{%
  \stepcounter{algorithmicH}
  \oldalgorithmic}
\renewcommand{\theHALG@line}{ALG@line.\thealgorithmicH.\arabic{ALG@line}}

\def\input@path{{./img/},{./data/},{./}}

\usetikzlibrary{decorations.markings}
\tikzset{
  nomorepostactions/.code={\let\tikz@postactions=\pgfutil@empty},
  mymark/.style 2 args={decoration={markings,
      mark= between positions 0.05 and 1 step (1/10)*\pgfdecoratedpathlength with{%
        \tikzset{#2,every mark}\tikz@options
        \pgfuseplotmark{#1}%
      },
    },
    postaction={decorate},
    /pgfplots/legend image post style={
      mark=#1,mark options={#2},every path/.append style={nomorepostactions}
    },
  },
}

\providecommand{\keywords}[1]{\textit{Keywords}: #1}

\newcommand{\Precon}{\mathcal{M}}
\newcommand{\Obj}{f}
\newcommand{\vect}[1]{\boldsymbol{#1}}

\newcommand{\Transpose}[1]{#1^\intercal}
\newcommand{\Hess}[1]{H#1}
\newcommand{\Grad}{g}
\newcommand{\vecbracks}[1]{\left[#1\right]}
\newcommand{\norm}[1]{\left\lVert#1\right\rVert}
\newcommand{\Krylov}[1]{\mathcal{K}_{#1}(A,r^{(1)})}
\newcommand{\OKrylov}[2]{\mathcal{K}_{#1}^O(#2)}

\author{Asbj{\o}rn Nilsen Riseth\thanks{riseth@maths.ox.ac.uk ---
    Mathematical Institute, University of Oxford, OX2 6GG, UK.}}
\title{Objective acceleration for unconstrained optimization}
\date{\today}
\begin{document}

\maketitle

\begin{abstract}
  Acceleration schemes can dramatically improve existing optimization procedures.
  In most of the work on these schemes, such as
  nonlinear Generalized Minimal Residual (N-GMRES), acceleration is
  based on minimizing the $\ell_2$ norm of
  some target
  on subspaces of $\mathbb{R}^n$.
  There are many numerical examples that show how accelerating general
  purpose
  and domain-specific optimizers with N-GMRES results in large
  improvements. We propose a natural modification to N-GMRES, which
  significantly improves the performance
  in a testing environment originally used to advocate
  N-GMRES. Our proposed approach, which we refer to as O-ACCEL
  (Objective Acceleration),
  is novel in that it minimizes an approximation to the \emph{objective function}
  on subspaces of $\mathbb{R}^n$.
  We prove that O-ACCEL reduces to the
  Full Orthogonalization Method for linear systems when the
  objective is quadratic, which differentiates our proposed approach
  from existing acceleration methods.
  Comparisons with L-BFGS and N-CG indicate the competitiveness
  of O-ACCEL. As it can be combined with domain-specific
  optimizers, it may also be beneficial in areas where L-BFGS or N-CG
  are not suitable.
\end{abstract}

\section{Introduction}
Gradient based optimization algorithms normally iterate based on
tractable approximations to the objective function at a particular
point.
Acceleration algorithms aim to combine the strengths of existing
solvers with information from previous iterates.
We propose an acceleration
scheme that can be used on top of existing optimization algorithms,
which generates a subspace from previous iterates, over which it
aims to optimize the objective function. We call the
algorithm O-ACCEL, short for Objective Acceleration.

Our idea closely resembles the work of \citet{sterck2013steepest},
which introduced the preconditioned nonlinear GMRES (N-GMRES)
algorithm for optimization. By using a more appropriate target to
accelerate the optimization than
N-GMRES does, we show, with numerical examples, how
O-ACCEL more efficiently accelerates the steepest
descent algorithm.
When optimizing an objective $\Obj$, N-GMRES is used as an accelerator from
the point of view of solving the nonlinear system $\nabla\Obj(x)=0$
which arises from the first-order condition of optimality.
It uses the idea of Krylov subspace acceleration from
\citet{washio1997krylov} and \citet{oosterlee2000krylov} for solving nonlinear
equations that arise from discretizations of partial differential
equations.
The name N-GMRES arises from the fact that steepest descent
preconditioned N-GMRES is equivalent to the standard GMRES procedure for linear
systems of equations \citep{washio1997krylov,sterck2013steepest}.
A similar idea, also arising from nonlinear equations, was described
in \citet{anderson1965iterative} in
\citeyear{anderson1965iterative}. See \citet{walker2011anderson} for a
note on
the similarities of the methods, and \citet{fang2009two} which
puts Anderson acceleration in the context of a Broyden-type
approximation  of the inverse Jacobian.
\citet{brune2015composing} show, with many numerical
examples,
that N-GMRES and Anderson acceleration can greatly improve convergence
on nonlinear systems,
when combined with an appropriate preconditioner (nonlinear solver).
In the setting of optimization, \citet{sterck2012nonlinear} and
\citet{sterck2016nonlinearly} show large improvements in convergence
by applying N-GMRES acceleration to the computation of tensor
decompositions.

More recently, \citet{damien2016regularized} have developed another
acceleration method for convex optimization denoted regularized
nonlinear acceleration (RNA), which \citet{cartis2017accelerating}
have extended to the nonconvex case.
Acceleration techniques differ from one another in several ways,
but, for convex quadratic objectives, the
Anderson, N-GMRES and Scieur et al.\ algorithms all coincide
\citep{cartis2017accelerating}.
These methods all minimize the $\ell_2$ norm of some objective in $\mathbb{R}^n$, the
space of the decision variable.
The proposed algorithm in this manuscript instead aims to minimize the
objective function over a subspace of $\mathbb{R}^n$.
We believe this is a natural target to accelerate against, especially
when the optimization procedure is seeking descent directions.
For convex, quadratic functions we prove that O-ACCEL with a steepest
descent preconditioner reduces to the full orthogonalization method
(FOM~\cite{saad2003iterative}), a Krylov subspace
procedure for solving linear systems. This differentiates our method
from the other acceleration techniques, which are related to the
GMRES algorithm for linear systems.

Due to the close similarity with the proposed algorithm and N-GMRES, this
manuscript focuses on numerical comparisons to N-GMRES under the same testing
conditions as used by \citet{sterck2013steepest}. On the test set
from \citet{sterck2013steepest}, our
acceleration scheme compares favourably to N-GMRES, as well as
implementations of the nonlinear conjugate gradient (N-CG) and limited-memory
Broyden--Fletcher--Goldfarb--Shanno (L-BFGS) methods \citep{nocedal2006numerical}.
Further tests on the CUTEst test problem set \citep{gould2015cutest}
show that L-BFGS is more applicable to these problems, however, O-ACCEL
again performs better than N-GMRES.

The manuscript is organized as follows. Motivation for the algorithm,
and discussion around it, is covered in \Cref{sec:algo_description}.
Numerical tests that show the efficiency of our proposed acceleration
procedure  applied to
steepest descent are presented in \Cref{sec:num_experiments}. We
conclude and discuss further potential work in \Cref{sec:conclusion}.

\section{Optimization acceleration with O-ACCEL}\label{sec:algo_description}
To fix notation, consider a twice continuously differentiable function
$\Obj\in C^2(\mathbb{R}^n)$ that
is bounded below and has at least one minimizer.
We aim to find a local minima of the optimization problem
\begin{equation}
  \min_{x\in\mathbb{R}^n}\Obj(x).
\end{equation}
Let $\Precon(\Obj,x)$ denote an optimization procedure for $\Obj$ with
initial guess $x\in\mathbb{R}^n$. This optimization procedure can, for example, be the
application of one steepest descent, or Newton, step. We will refer to
$\Precon$ as the preconditioner, because it is applied in the same fashion
as a right preconditioner for iterative procedures of linear systems
\citep{brune2015composing}.
Given a sequence of previously explored iterates
$x^{(1)},\dots,x^{(k)}$, and
a proposed new guess $x^P = \Precon(\Obj,x^{(k)})$, we will try
to accelerate the next iterate $x^{(k+1)}$ towards a minimizer.
Define
\begin{equation}
  \OKrylov{k}{x^P}=\mbox{span}\{x^{(1)}-x^P,\dots,x^{(k)}-x^P\}.
\end{equation}
The acceleration step aims to minimize $\Obj$ over the
subset $x^P+\OKrylov{k}{x^P}$, which can be interpreted as a
generalisation from a line search to a hyperplane search.
Let $\alpha\in \mathbb{R}^k$, and set
\begin{equation}
  x^A(\alpha) = x^P + \sum_{j=1}^k\alpha_j(x^{(j)}-x^P).
\end{equation}
Note that, when $k=1$, minimizing $\Obj$ over $\OKrylov{k}{x^P}$ is
equivalent to the standard line search problem
of minimizing $\lambda \mapsto \Obj(x^{(1)}+\lambda (x^P-x^{(1)}))$.
The first-order condition for $\alpha$ to be a minimizer of the function
$\alpha\mapsto \Obj(x^A(\alpha))$ is
\begin{equation}
  \nabla_{\alpha} \Obj(x^P+{\textstyle\sum_{j=1}^k}\alpha_j(x^{(j)}-x^P)) =0.
\end{equation}
Define the gradient $\Grad(x)=\nabla_x\Obj(x)$. For $l=1,\dots,k$,
\begin{equation}\label{eq:deriv_objective_alpha}
  \frac{\partial }{\partial \alpha_l}
  \Obj({\textstyle x^P+\sum_{j=1}^k\alpha_j(x^{(j)}-x^P)})
  = \Transpose{{\Grad\left({\textstyle x^P+\sum_{j=1}^k\alpha_j(x^{(j)}-x^P)}\right)}}
  (x^{(l)}-x^P),
\end{equation}
where superscript $\intercal$ denotes the transpose.
The O-ACCEL algorithm aims to linearize the first-order condition $\nabla_\alpha \Obj(x^A(\alpha))=0$ in
the following way. Let $\Hess{}(x)$ denote the Hessian of $\Obj$
at $x$. By linearizing $\alpha\mapsto\nobreak\Grad(x^A(\alpha))$, we get
\begin{align}\label{eq:grad_xA_approx}
  \begin{split}
    \Grad\left({\textstyle x^P+\sum_{j=1}^k\alpha_j(x^{(j)}-x^P)}\right)
    &\approx
    \Grad(x^P) + \Hess{}(x^P)
    \sum_{j=1}^k\alpha_j(x^{(j)}-x^P)\\
    &= \Grad(x^P)
    +\Hess{}(x^P)(\vect{X}-\vect{X}^P)\alpha,
  \end{split}
\end{align}
where we use the matrices
$\vect{X} = \vecbracks{x^{(1)},\dots,x^{(k)}}\in\mathbb{R}^{n\times k}$ and
$\vect{X}^P = \vecbracks{x^P,\dots,x^P}\in\mathbb{R}^{n\times k}$.
Given this linearization we aim to find an $\alpha\in\mathbb{R}^k$
that approximately satisfies the first-order condition. 
We can do this by
combining~\eqref{eq:deriv_objective_alpha} and~\eqref{eq:grad_xA_approx}, and then
look for an
$\alpha\in\mathbb{R}^k$ that solves
\begin{equation}
  \Transpose{\alpha}\Transpose{{(\vect{X}-\vect{X}^P)}}\Hess{}(x^P)
  (x^{(l)}-x^P) = -\Transpose{{\Grad(x^P)}}(x^{(l)}-x^P),\qquad l=1,\dots,k.
\end{equation}
In matrix form, the system of equations becomes
\begin{equation}\label{eq:ngmres_eq_hess}
  \Transpose{{\left(\vect{X}-\vect{X}^P\right)}}\Hess{}(x^P)
  \left(\vect{X}-\vect{X}^P\right)
  \alpha
  = -\Transpose{{(\vect{X}-\vect{X}^P)}}{\Grad(x^P)}.
\end{equation}

There are cases where we may not wish to compute the Hessian of
$\Obj$ explicitly, for example, if $\Precon$ does not use it.
We can instead use an approximation $\tilde{\Hess{}}(x^P)$ of the Hessian
$\Hess{}(x^P)$, or its action on vectors in
$\OKrylov{k}{x^P}$.
The iterative Hessian approximation algorithms that are used in
quasi-Newton methods can provide one avenue of research.
In the numerical experiments provided in this manuscript, we instead
focus on approximating the action of the Hessian on
$\OKrylov{k}{x^P}$ to first order by
\begin{equation}\label{eq:approx_hess}
  \Hess{}(x^P) (x^{(l)}-x^P)
  \approx \Grad(x^{(l)})-\Grad(x^P).
\end{equation}
Let $\Grad(\vect{X}) =
\vecbracks{\Grad(x^{(1)}),\dots,\Grad(x^{(k)})}$, and
define $\Grad(\vect{X}^P)$ similarly.
This gives a second approximation to the first-order conditions,
\begin{equation}\label{eq:ngmres_eq_grad}
  \Transpose{{(\vect{X}-\vect{X}^P)}}\left(
    \Grad(\vect{X})-\Grad(\vect{X}^P)
  \right)
  \alpha
  = -\Transpose{{(\vect{X}-\vect{X}^P)}}{\Grad(x^P)}.
\end{equation}
In this manuscript, we investigate the performance of the objective-based
acceleration using~\eqref{eq:ngmres_eq_grad}.

To contrast our work with the N-GMRES optimization algorithm in
\citet{sterck2013steepest}, minimizing the
$\ell_2$ norm of the approximation of $\Grad(x^A)$
established from~\eqref{eq:grad_xA_approx} and~\eqref{eq:approx_hess} results in
the linear least squares problem
\begin{equation}
  \min_{\alpha\in\mathbb
    {R}^k}\Bigl\|\Grad(x^P)+\sum_{j=l}^k\alpha_l(\Grad(x^{(l)})-\Grad(x^P))\Bigr\|_2.
\end{equation}
Its solution can be found from the normal equation
\begin{equation}\label{eq:ngmres_lin_sys}
  \Transpose{{\left(\Grad(\vect{X})-\Grad(\vect{X}^P)\right)}}\left(
    \Grad(\vect{X})-\Grad(\vect{X}^P)
  \right)
  \alpha
  = -\Transpose{{\left(\Grad(\vect{X})-\Grad(\vect{X}^P)\right)}}{\Grad(x^P)}.
\end{equation}

We argue that the O-ACCEL algorithm is more appropriate for an
optimization problem than N-GMRES. 
When we are restricted to subsets of the
decision space, reduction in the value of the objective is a better
indicator of moving towards a minimizer than reduction in the gradient norm. In
effect, N-GMRES ignores the extra information provided by $\Obj$.
This is better illustrated in the case when $k=1$, where it is standard
to perform a line search on the objective rather than the gradient norm.

\subsection{Algorithm}
The proposed acceleration procedure,
which we call O-ACCEL, is described in \Cref{alg:ngmreso}.
The number of stored previous iterates $w$ denotes the history size.
Setting an upper bound on the history size can be necessary due to storage
constraints, or to prevent the local approximations
of~\eqref{eq:grad_xA_approx} and~\eqref{eq:approx_hess} from using iterates far away
from
$x^P$. If the direction from $x^P$ to the accelerated step $x^A$ is
not a descent direction, it indicates that the linearized
approximation around $x^P$ is bad for the currently stored iterates.
For simplicity, we therefore choose to reset the history size to $w=1$
when we encounter such cases.
\begin{algorithm}[htb]
  \caption{The O-ACCEL algorithm}\label{alg:ngmreso}
  \begin{algorithmic}[1]
    \Procedure{OACCEL}{$x^{(1)},\dots,x^{(w)}$}
    \State $x^P\gets \Precon(\Obj,x^{(w)})$
    \State Approximate $\Hess{}(x^P)\approx \tilde{\Hess{}}$, or
    its action
    \State $A\gets
    \Transpose{{\left(\vect{X}-\vect{X}^P\right)}}\tilde{\Hess{}}
    \left(\vect{X}-\vect{X}^P\right)$
    \hfill \Cref{sec:num_experiments}: $A\gets\Transpose{{(\vect{X}-\vect{X}^P)}}\left(
      \Grad(\vect{X})-\Grad(\vect{X}^P)
    \right)$
    \State $b\gets -\Transpose{{(\vect{X}-\vect{X}^P)}}{\Grad(x^P)}$
    \State Solve  $A\alpha = b$
    \State $x^A \gets x^P + \sum_{j=1}^w\alpha_j(x^{(j)}-x^P)$
    \If{$x^A-x^P$ is a descent direction}\label{algline:descent_dir}
    \State $x^{(w+1)}\gets \textnormal{linesearch}(x^P+\lambda(x^A-x^P))$\label{algli:xp_xa_ls}
    \State \texttt{reset} $\gets$ \texttt{false}
    \Else
    \State $x^{(w+1)}\gets x^P$
    \State \texttt{reset} $\gets$ \texttt{true}
    \EndIf
    \State \Return $(x^{(w+1)},\texttt{reset})$
    \EndProcedure
  \end{algorithmic}
\end{algorithm}

To prevent re-computation of $\Grad(x^{(j)})$ for $j=1,\dots,w$ in each application of
the procedure, we store these vectors for later use.
The computational cost of the algorithm is approximately the same as
$w$-history L-BFGS with two-loop recursion \citep{sterck2013steepest}. 
In terms of storage, O-ACCEL and L-BFGS both store $2w$ vectors
of size $n$. In addition, our implementation of O-ACCEL, as described
in \Cref{alg:ngmreso_impl} below, reduces the number of flops
required by storing a $w\times w$ matrix of previously calculated
values. For the numerical experiments we have used $w=20$, in
accordance with \citet{sterck2013steepest}.
It was, however, shown by \citet{sterck2013steepest} that N-GMRES can
already provide good results with $w=3$.
Tests using O-ACCEL with $w=5$, although not included here, provide
almost as good results as reported in \Cref{sec:num_experiments}.
Note that, if the Hessian is sparse, it may be more storage efficient
to find $\alpha$ from the linear system in~\eqref{eq:ngmres_eq_hess}
than using a large $w$.

\begin{remark}
  The RNA approach suggested by \citet{damien2016regularized}
  is called separately from the iterations by the optimizer, when
  judged appropriate. This contrasts with \Cref{alg:ngmreso}, where
  the acceleration happens at each iteration.
  The O-ACCEL
  acceleration can be applied separately in the same fashion
  as in \citet{damien2016regularized},
  however, this is not considered in this manuscript.
\end{remark}

\subsection{O-ACCEL as a full orthogonalization method (FOM)}
The optimality condition~\eqref{eq:deriv_objective_alpha} for the
function $\alpha\mapsto \Obj(x^A(\alpha))$ is
$\Transpose{\Grad(x^A)}(x^{(l)}-x^P)=0$, for $l=1,\dots,k$.
Hence, we look for $x^A\in x^P+\OKrylov{k}{x^P}$ so that
$\Grad(x^A)\perp \OKrylov{k}{x^P}$.
This condition reduces to FOM \citep{saad2003iterative} when
$\Grad(x)$ is linear and $\Precon(\Obj,x)$ is a steepest descent
algorithm.
When the Hessian is symmetric
positive-definite, FOM is mathematically equivalent to the conjugate
gradient method. We can therefore think of O-ACCEL as a
N-CG method that approximates the orthogonality condition with a
larger history size.

The FOM is an iterative procedure for solving a linear system $Ax=b$.
With initial guess $x^{(1)}$ and residual $r^{(1)}=b-Ax^{(1)}$,
define the Krylov subspace
\begin{equation}
  \Krylov{k}=\mbox{span} \{r^{(1)},Ar^{(1)},\dots,A^{k-1}r^{(1)}\}.
\end{equation}
The iterate $x^{(k+1)}$ of FOM is an element in
$x^{(1)}+\Krylov{k}$ such that
$b-Ax^{(k+1)}\perp \Krylov{k}$.

For convex, quadratic objectives
$\Obj(x)=\frac{1}{2}\Transpose{x}Ax-\Transpose{x}b$,
the gradient $\Grad(x)=Ax-b$ is linear and
the optimum must satisfy the equation $Ax=b$.
The residuals $r^{(k)}=b-Ax^{(k)}$ are equal to the negative gradient
$-\Grad(x^{(k)})$. Therefore, O-ACCEL with a steepest descent
preconditioner yields
$x^P = \Precon(\Obj,x^{(k)}) =  x^{(k)}+\lambda^{(k)}r^{(k)}$ for some
$\lambda^{(k)}>0$.

\begin{theorem}\label{thm:oaccel_fom}
  Let $\Precon$ be a steepest descent preconditioner and
  $\Obj(x)=\frac{1}{2}\Transpose{x}Ax-\Transpose{x}b$.
  Let the O-ACCEL algorithm
  take the step $x^{(w+1)}=x^A$ in
  \Cref{algli:xp_xa_ls} of \Cref{alg:ngmreso}.
  Then the iterates of the O-ACCEL algorithm
  form the FOM sequence of the linear system $Ax=b$.
\end{theorem}

We shall shortly prove the theorem after deriving new expressions for $\Krylov{k}$.
First, note that for any $x$, a reordering of terms can show that
\begin{align}\label{eq:okrylov_2}
  \OKrylov{k}{x}
  &=
    \mbox{span}\{x-x^{(k)},x^{(k-1)},\dots,x^{(2)}-x^{(1)}\},\\
  x+\OKrylov{k}{x}
  &= x^{(1)} +
    \mbox{span}\{x-x^{(k)},x^{(k-1)},\dots,x^{(2)}-x^{(1)}\}.
    \label{eq:x_okrylov_2}
\end{align}
This motivates the next lemma, which connects the space on the right
hand side of~\eqref{eq:okrylov_2} to $\Krylov{k+1}$.
\begin{lemma}\label{lem:fom_xp}
  Let $x^{(1)},\dots,x^{(k)}$ be a given sequence of FOM iterates for
  a linear system $Ax=b$.
  Assume that
  $\mbox{span}(x^{(k)}-x^{(k-1)},\dots,x^{(2)}-x^{(1)})=\Krylov{k}$,
  and let $x^P = x^{(k)}+\lambda r^{(k)}$ for some
  $\lambda>0$.
  Then,
  \begin{equation}
    \mbox{span}\{x^P-x^{(k)},x^{(k)}-x^{(k-1)},\dots,x^{(2)}-x^{(1)}\}
    = \Krylov{k+1}.
  \end{equation}
\end{lemma}
\begin{proof}
  By definition of $x^P$ and the properties of the FOM sequence,
  \begin{equation}\label{eq:xp_xk_rk}
    x^{P}-x^{(k)} = \lambda r^{(k)}\perp \Krylov{k}.
  \end{equation}
  As $x^{(k)}\in x^{(1)}+\Krylov{k}$, we have
  $r^{(k)}\in\Krylov{k+1}$ because
  \begin{align}
    r^{(k)} &\in b - A(x^{(1)} +\Krylov{k})
    = r^{(1)} - A\Krylov{k}
    \in \Krylov{k+1}.
  \end{align}
  Therefore, $\mbox{span}\{r^{(k)},\Krylov{k}\}=\Krylov{k+1}$.
  This equality yields the result by replacing $r^{(k)}$ and
  $\Krylov{k}$ with~\eqref{eq:xp_xk_rk}
  and $\mbox{span}\{x^{(k)}-x^{(k-1)},\dots,x^{(2)}-x^{(1)}\}$.
\end{proof}

\begin{proof}[Proof of \Cref{thm:oaccel_fom}]
  We prove the result by induction on the
  sequence $x^{(1)},\dots,x^{(k)}$ arising from the O-ACCEL algorithm.
  Let $k=2$, then
  \begin{align}
    x^{(2)}
    &= x^P + \alpha^{(1)}(x^{(1)}-x^P)\\
    &= x^{(1)} + \lambda^{(1)}r^{(1)} -
      \alpha^{(1)}\lambda^{(1)}r^{(1)}
      \in x^{(1)} + \Krylov{1}.
  \end{align}
  and so $\mbox{span}\{x^{(2)}-x^{(1)}\} = \Krylov{1}$.
  From~\eqref{eq:deriv_objective_alpha} the residual
  $b-Ax^{(2)}\perp x^P-x^{(1)} = \lambda^{(k)}r^{(1)}$, and thus
  $x^{(2)}$ is the second FOM iterate. This establishes the base case
  for the induction proof.

  The inductive step follows from \Cref{lem:fom_xp}
  together with~\eqref{eq:okrylov_2} and~\eqref{eq:x_okrylov_2},
  and hence proves that
  the O-ACCEL iterates are the FOM iterates
  for $Ax=b$.
\end{proof}

\begin{remark}
  The connection to the FOM differentiates O-ACCEL from
  N-GMRES, Anderson acceleration, and RNA, which
  reduce to GMRES for quadratic objectives.
\end{remark}

\section{Numerical experiments}\label{sec:num_experiments}
In order to investigate the performance of the proposed algorithm, we
implement it with two preconditioners $\Precon$. The first is steepest
descent with line search, and the second is steepest descent with a
fixed step length.
They are compared to the N-GMRES algorithm with the
same preconditioners, and implementations of the nonlinear conjugate
gradient (N-CG) variant with the Polak-Ribi\`{e}re update formula, and the
two-loop recursion version of the limited-memory
Broyden--Fletcher--Goldfarb--Shanno (L-BFGS) method
\citep{nocedal2006numerical}.
The test problems considered in
Sections~\ref{subsec:test_problems_sterck} to
\ref{subsec:tensor_optim} are the same eight problems that were used
in \citet{sterck2013steepest} to advocate N-GMRES. We also include
experiments from \num{33} CUTEst problems to further test the
applicability of the algorithms.
The results are presented in the
form of performance profiles, as introduced by
\citet{dolan2002benchmarking}, based on the number of
function/gradient evaluations.

The main focus of this manuscript is to compare the performance of the
proposed algorithm to the N-GMRES algorithm. To this end, we have used
the MATLAB implementation of this algorithm, available
online.\footnote{\url{http://www.hansdesterck.net/Publications-by-topic/nonlinear-preconditioning-for-nonlinear-optimization}}
The O-ACCEL implementation, and the rest of the code required to
generate the test result data, is also made available by the
author.\footnote{\url{https://github.com/anriseth/objective_accel_code}}
Our implementation of O-ACCEL follows the exact same steps, only
replacing the calculations needed to solve the N-GMRES system
in~\eqref{eq:ngmres_lin_sys}
with those of the linear system in~\eqref{eq:ngmres_eq_grad}.
The implementation is detailed in \Cref{alg:ngmreso_impl}. It closely
follows the instructions from
\citet{washio1997krylov}, including a regularization for the
linear system.

The regularization is used prevent the direct linear solver we use to
find $\alpha$ from crashing when $A$ is ill-conditioned or singular,
which can happen if the vectors $g(x^{(k)})-g(x^P)$ are linearly
dependent.
Let $A\in\mathbb R^{w\times w}$ denote the system matrix $\Transpose{{(\vect{X}-\vect{X}^P)}}\left(
  \Grad(\vect{X})-\Grad(\vect{X}^P)
\right)$.
Then, for some tolerance $\epsilon_0>0$, set
$\epsilon=\epsilon_0\cdot\max{\{A_{ii}\}}_{i=1}^w$. The $\max$ term
is used to scale the regularization in accordance with the
optimisation problem.
With $I\in\mathbb{R}^{w\times w}$ the identity matrix, we solve
the linear problem
\begin{equation}
  (A+\epsilon I)\alpha = b,
\end{equation}
rather than the linear problem $A\alpha=b$ as defined in
\Cref{alg:ngmreso}.
This is a Tikhonov type regularization~\cite{neumaier1998solving},
often employed to regularize ill-conditioned problems.
\citet{washio1997krylov} shows that the error in the resulting
$\alpha$ is negligible for the N-GMRES
problem~\eqref{eq:ngmres_lin_sys} provided $\epsilon$ is much smaller
than the smallest non-zero eigenvalue of the system matrix.
The error for the O-ACCEL system can be analysed within a general
Tikhonov regularization framework, see, for example,
\citet{neumaier1998solving}. We do not investigate the impact of the
regularization parameter further in this manuscript, and use the
value $\epsilon_0=10^{-14}$ that was used in the N-GMRES code by
\citet{sterck2013steepest}.

\begin{algorithm}[htb]
  \caption{Implementation of O-ACCEL
    algorithm. Indentation and curly brackets denote scope.}\label{alg:ngmreso_impl}
  \hspace*{\algorithmicindent} \textbf{Input}:
  $\Obj$, $\Grad$, $\Precon$, $x$, $w_{\textnormal{max}}$, $\epsilon_0$, tolerance
  description\\
  \hspace*{\algorithmicindent} \textbf{Output}:
  $x$ satisfying tolerance description
  \begin{algorithmic}[1]
    \While{Not reached tolerance}
    \State $\texttt{x}_1 \gets x$ ; $\texttt{r}_1\gets \Grad(x)$ ; $\texttt{q}_{11}\gets
    \Transpose{x}\texttt{r}_1$
    \State $w\gets 1$ ; $k\gets 0$ ; \texttt{reset} $\gets$ \texttt{false}
    \While{\texttt{reset} is \texttt{false}}
    \State $k\gets k+1$
    \State $x \gets \Precon(\Obj,x)$ ; $r\gets \Grad(x)$
    \If{reached tolerance}
    \State \textbf{break}
    \EndIf
    \State $\eta \gets \Transpose{x}{r}$
    \State
    \textbf{for} $i=1,\dots,w$
    \{ $\xi^{(1)}_i\gets
    \Transpose{\texttt{x}_i}r$ ; $
    \xi^{(2)}_i\gets
    \Transpose{x}\texttt{r}_i$ ; $
    \texttt{b}_i\gets \eta-\xi^{(1)}_i$ \}
    \State
    \textbf{for} $i=1,\dots,w$ \{
    \textbf{for} $j = 1,\dots, w$
    \{ $\texttt{A}_{ij}\gets \texttt{q}_{ij}-\xi^{(1)}_i-\xi^{(2)}_j+\eta$ \} \}
    \State $\epsilon \gets \epsilon_0 \cdot \max\{\texttt{A}_{11,}\dots,\texttt{A}_{ww}\}$
    \State Solve
    $
    \begin{pmatrix}
      \texttt{A}_{11}+\epsilon&\cdots&\texttt{A}_{1w}\\
      \vdots&\ddots&\vdots\\
      \texttt{A}_{w1}&\cdots&\texttt{A}_{ww}+\epsilon
    \end{pmatrix}
    \begin{pmatrix}
      \alpha_1\\\vdots\\\alpha_w
    \end{pmatrix}
    = \begin{pmatrix}
      \texttt{b}_1\\\vdots\\\texttt{b}_w
    \end{pmatrix}$
    \State $x^A\gets  x +  \sum_{i=1}^w \alpha_i(\texttt{x}_i-x)$
    \State $d\gets x^A-x$
    \If{$\Transpose{d}r \geq 0$}
    \State \texttt{reset} $\gets$ \texttt{true}
    \Else
    \State $x\gets \textnormal{linesearch}(x+\lambda d)$
    \State $w\gets \min(w+1,w_{\textnormal{max}})$
    \State $j \gets (k \mod w_{\textnormal{max}}) + 1 $
    \State $\texttt{x}_j\gets x$
    \State $\texttt{r}_j\gets \Grad(x)$
    \State \textbf{for} $i=1,\dots,w$
    \{ $\texttt{q}_{ij}\gets \Transpose{\texttt{x}_i}\texttt{r}_j$ ; $\texttt{q}_{ji}\gets \Transpose{\texttt{x}_j}\texttt{r}_i$ \}
    \EndIf
    \EndWhile
    \EndWhile
  \end{algorithmic}
\end{algorithm}

For the remainder of the section, we present the test problems,
provide details for the parameter choices, and discuss the test results.

\subsection{Test problems from \citeauthor{sterck2013steepest}}\label{subsec:test_problems_sterck}
We describe the seven test problems from \citet{sterck2013steepest}.
All the functions are defined as $\Obj:\mathbb{R}^n\to\mathbb{R}$, and
the matrices mentioned are all in $\mathbb{R}^{n\times n}$.

Problem A. Quadratic objective function with symmetric, positive
definite diagonal matrix $D$,
\begin{align}
  \begin{split}
    \Obj(x)&= {\textstyle\frac{1}{2}}\Transpose{{(x-x^*)}}D(x-x^*),\text{ where }\\
    D&=\mbox{diag}(1,2,\dots,n), \text{ and }\\
    x^*&=\vecbracks{1,\dots,1}.
  \end{split}
\end{align}
The minimizer $x^*$ of Problem A is unique, with $\Obj(x^*)=0$. The
gradient is given by $\Grad(x)=D(x-x^*)$.

Problem B. Problem A with paraboloid coordinate transformation,
\begin{align}
  \begin{split}
    \Obj(x)&= {\textstyle\frac{1}{2}}\Transpose{{y(x-x^*)}}Dy(x-x^*),\text{ where }\\
    D&=\mbox{diag}(1,2,\dots,n),\\
    x^*&=\vecbracks{1,\dots,1}, \text{ and }\\
    y_1(z)&=z_1 \text{ and } y_j(z)=z_j-10z_1^2 \quad (i=2,\dots,n).
  \end{split}
\end{align}
The minimizer is again $x^*$, with $\Obj(x^*)=0$. The gradient is
$\Grad(x) = Dy(x-x^*) - 20(x_1-x_1^*) \times\newline \left( \sum_{j=2}^n{(Dy(x-x^*))}_j
\right)
\Transpose{\vecbracks{1,0,\dots,0}}$.

Problem C. Problem B with a random nondiagonal matrix $T$ with condition
number $n$,
\begin{align}
  \begin{split}
    \Obj(x)&={\textstyle\frac{1}{2}}\Transpose{{y(x-x^*)}}Ty(x-x^*),\text{ where }\\
    x^*&=\vecbracks{1,\dots,1},\\
    y_1(z)&=z_1 \text{ and } y_j(z)=z_j-10z_1^2 \quad (i=2,\dots,n),
    \text { and}\\
    T&=Q\,\mbox{diag}(1,2,\dots,n)\,\Transpose{Q},
  \end{split}
\end{align}
where $Q$ is a random orthogonal matrix.
As in Problems A and B, the minimizer is $x^*$ with $\Obj(x^*)=0$. The
gradient is
$\Grad(x) = Ty(x-x^*) - 20(x_1-x_1^*) \times \left( \sum_{j=2}^n{(Ty(x-x^*))}_j
\right)
\Transpose{\vecbracks{1,0,\dots,0}}$.

Problem D. Extended Rosenbrock function, Problem (21) from
\citet{more1981testing},
\begin{align}
  \begin{split}
    \Obj(x)&={\textstyle\frac{1}{2}}\sum_{j=1}^n {t_j(x)}^2,\text{ where $n$ is even,}\\
    t_j&=10(x_{j+1}-x_j^2)\qquad\text{($j$ odd), and}\\
    t_j&=1-x_{j-1}\qquad\qquad\text{($j$ even).}
  \end{split}
\end{align}
The unique minimum $\Obj(x^*)=0$ is attained at $x^*=\vecbracks{1,\dots,1}$.
The derivative can be computed using $g_k(x)=\sum_{j=1}^n t_j
\frac{\partial t_j}{\partial x_k}$, ($k=1,\dots,n$).
Gradients for Problems E-G can be computed in similar fashion.

Problem E. Extended Powell singular function, Problem (22) from
\citet{more1981testing},
\begin{align}
  \begin{split}
    \Obj(x)&={\textstyle\frac{1}{2}}\sum_{j=1}^n {t_j(x)}^2,\text{
      where $n$ is a multiple of
      4,}\\
    t_{4j-3} &= x_{4j-3}+10x_{4j-2},\\
    t_{4j-2} &= \sqrt{5}(x_{4j-1}-x_{4j}),\\
    t_{4j-1} &= {(x_{4j-2}-2x_{4j-1})}^2,\\
    t_{4j} &= \sqrt{10}{(x_{4j-3}-x_{4j})}^2
    \qquad \text{ for $j = 1,\dots,n/4$.}
  \end{split}
\end{align}
The unique minimum $\Obj(x^*)=0$ is attained at $x^*=0$.

Problem F. The Trigonometric function, Problem (26) from
\citet{more1981testing},
\begin{align}
  \begin{split}
    \Obj(x)&={\textstyle\frac{1}{2}}\sum_{j=1}^n {t_j(x)}^2,\text{ where}\\
    t_j &= n  + j(1-\cos x_j) - \sin x_j -\sum_{i=1}^n\cos(x_i).
  \end{split}
\end{align}
The unique minimum $\Obj(x^*)=0$ is attained at $x^*=0$.
Note that in \citet{sterck2013steepest}, a minus sign is used in front
of $j(1-\cos
x_j)$. We follow the original formulation of \citet{more1981testing}.

Problem G. Penalty function I, Problem (23) from
\citet{more1981testing},
\begin{align}
  \begin{split}
    \Obj(x)&={\textstyle\frac{1}{2}}\Bigl(  {t_0(x)}^2 +
    \sum_{j=1}^n {t_j(x)}^2\Bigr),\text{ where}\\
    t_0 &= -0.25 + \sum_{j=1}^nx_j^2, \text{ and}\\
    t_j &= \sqrt{10^{-5}}(x_j-1)\qquad (j=1,\dots,n).
  \end{split}
\end{align}
The minimum is not known explicitly for Problem G, and depends on the
value of $n$.

\subsection{Experiment design}\label{subsec:experiment_design}
We test the N-GMRES and O-ACCEL algorithms with two steepest descent
preconditioners
\begin{align}
  \Precon_{\text{Z}}(\Obj,x) &= x-\lambda_{\text{Z}} \frac{\Grad(x)}{\norm{\Grad(x)}_2}, \qquad
                            \text{with Z = A, B, and}\nonumber\\
  \lambda_{\text{A}} &= \text{determined by line search},\\
  \lambda_{\text{B}} &= \min(\delta,\norm{\Grad(x)}_2).
\end{align}
Thus, the two preconditioners only differ in the choice of step
length. Option A employs a globalizing strategy with a chosen line
search, whilst option B takes a predetermined step length.
By choosing a short, predetermined step length $\delta>0$, we expand
the subspace to search for $\alpha$ and stay close to the
previous iterate $x^{(k)}$, hopefully improving the linearizations
in~\eqref{eq:grad_xA_approx} and~\eqref{eq:approx_hess}.
For the experiments, we use the line search algorithm by
\citet{more1994line}, which satisfies the Wolfe conditions
\citep{nocedal2006numerical}.
It is both employed for $\Precon_{\text{A}}$, and in the line search
$x^P+\lambda(x^A-x^P)$  between the preconditioned step $x^P$ and the
accelerated step $x^A$ of the N-GMRES and O-ACCEL routines.

To closely follow the testing conditions of
\citet{sterck2013steepest}, we use the N-CG, L-BFGS and
Mor\'{e}-Thuente line search implementations from the Poblano toolbox
by \citet{dunlavy2010poblano}.
These may not be state of the art implementations, however, the main
focus of this manuscript is to investigate the performance of the
N-GMRES and O-ACCEL algorithms. Future work will include testing the
O-ACCEL algorithm with appropriate preconditioners on more
comprehensive test sets, against state of the art implementations of
gradient based optimization algorithms.

All optimization procedures employ the Mor\'{e}-Thuente line search
with the following options: decrease tolerance $c_1=10^{-4}$
and curvature tolerance $c_2=0.1$
for the Wolfe conditions, starting step length $\lambda=1$, and a
maximum of \num{20} $\Obj/\Grad$ evaluations.
The N-GMRES and O-ACCEL history lengths are set to
$w_{\textnormal{max}}=20$, and the regularization parameter is set to
$\epsilon_0=10^{-12}$.
For $\Precon_{\text{B}}$, the fixed step length is set to $\delta=10^{-4}$.
The L-BFGS history size is set to \num{5}. Larger history sizes were
found by \citet{sterck2013steepest} to be harmful for the L-BFGS
performance on this test set.

Note that our choice of curvature tolerance $c_2=0.1$ is different from
\citet{sterck2013steepest}, where $c_2=0.01$ was used. There are two
reasons for this.
First, our choice is often used in
practice, see \citet[Ch.~3.1]{nocedal2006numerical}, and it
reduces the number of function evaluations for all the solvers
considered.
Second, we are interested in comparing the outer solvers, however, smaller values of $c_2$ moves work from the outer solvers to the
line search algorithms.

We test Problem A-C for both problem sizes $n=100$ and $n=200$.
Problem D is tested with
$n=\num{500},\num{1000},\num{50000},\num{100000}$.
Problem E with $n=\num{100},\num{200},\num{50000},\num{100000}$.
Problem F is called with $n=\num{200},\num{500}$, and finally, Problem
G with $n=\num{100},\num{200}$.
Each combination of problem and problem size is run \num{1000} times, with the
components of the initial guess drawn uniformly random from the interval
$[0,1]$. For Problem C, each instance of the problem generates a new,
random, orthogonal matrix $Q$.
This results in \num{18000} individual tests for the comparison.
To evaluate performance, we count the number of  objective
evaluations required for the algorithms to reach an iterate $x$
such that $\Obj(x)-\Obj^* < 10^{-10}(\Obj(x^{(0)})-\Obj^*)$. A solver run is labelled as failed
if it does not reach tolerance within \num{1500} iterations.
The minimum value $\Obj^*$ is
known for Problems A-F, however for Problem G we estimate $\Obj^*$
using the lowest value attained across all the optimisation
procedures.
The results on the collection of \num{18000} test instances are
discussed in \Cref{subsec:perf_prof}, whilst the Appendix provides
tables of results on the individual problems and problem sizes.

Note that our reporting of the numerical experiments differ from
those of \citet{sterck2013steepest} in two ways: First, we run each problem
combination \num{1000} times, instead of \num{10} times. Second, we
evaluate the results based on performance profiles and tables of
quantiles, instead of solely reporting the average number of
evaluations to reach tolerance.
We believe the high number of test runs is important for more
consistent values of the statistics reported in the Appendix across computers, further stabilised by using
quantiles rather than averages.

\subsection{Performance profiles}\label{subsec:perf_prof}
In order to evaluate the performance of optimizers on test sets with
problems of varying size  and difficulty,
\citet{dolan2002benchmarking} proposed the use of performance
profiles.
For completeness, we first define the performance profile for our chosen
metric of objective evaluations.
Let $\mathcal{P}$ denote the test set of the $n_p=\num{18000}$
problems, and $n_s$ the number of solvers.
For each problem $p\in\mathcal{P}$, and solver $s$, define
\begin{equation}
  t_{p,s} = \text{number of $\Obj$ evaluations required to reach
    tolerance.}
\end{equation}
In the numerical tests we say that the solver has reached tolerance
for the problem when the relative decrease in the objective value is
at least $10^{-10}$, that is
\begin{equation}\label{eq:def_performance_measure}
  t_{p,s} = \min\{k\geq 1 \mid \Obj(x^{(k)})-\Obj^* < 10^{-10}(\Obj(x^{(0)})-\Obj^*)\}.
\end{equation}
\begin{remark}
  Note that the numbers of objective and gradient calls are the same for
  each of the optimizers considered in this manuscript. This is due to
  the use of the  Mor\'{e}-Thuente line search algorithm.
\end{remark}

Let $\underline{t}_p$ denote the lowest number of $\Obj$ evaluations needed
to reach tolerance for problem $p$ across all the solvers,
\begin{equation}
  \underline{t}_p=\min\{t_{p,s}\mid 1\leq s\leq n_s\}.
\end{equation}
The performance ratio measures the performance on problem $p$ by
solver $s$, as defined by
\begin{equation}
  \rho_{p,s}=t_{p,s}/\underline{t}_p.
\end{equation}
The value is bounded below by $1$, and $\rho_{p,s}=1$ for at least one
solver $s$.
If solver $s$ does not solve problem $p$, then we set $\rho_{p,s}=\infty$.
We define the performance profile $p_s:[1,\infty)\to[0,1]$, for solver
$s$, by
\begin{equation}\label{eq:perf_profile}
  p_s(\tau) = \frac{1}{n_p}\mbox{size}\left\{p\in\mathcal{P}\mid \rho_{p,s}\leq \tau\right\}.
\end{equation}
The performance profile for a solver $s$ can be viewed as an
empirical, cumulative ``distribution'' function representing the
probability of the solver $s$ reaching tolerance within a ratio $\tau$
of the
fastest solver for each problem.
In particular, $p_s(1)$ gives the proportion of problems for which
solver $s$ performed best. For large values of $\tau$, the performance
profile $p_s(\tau)$ indicates robustness, that is, what proportion of all
the test problems were solved by the solver.

\begin{figure}[htb]
  \centering
  \includegraphics[width=0.8\textwidth,axisratio=1.7]{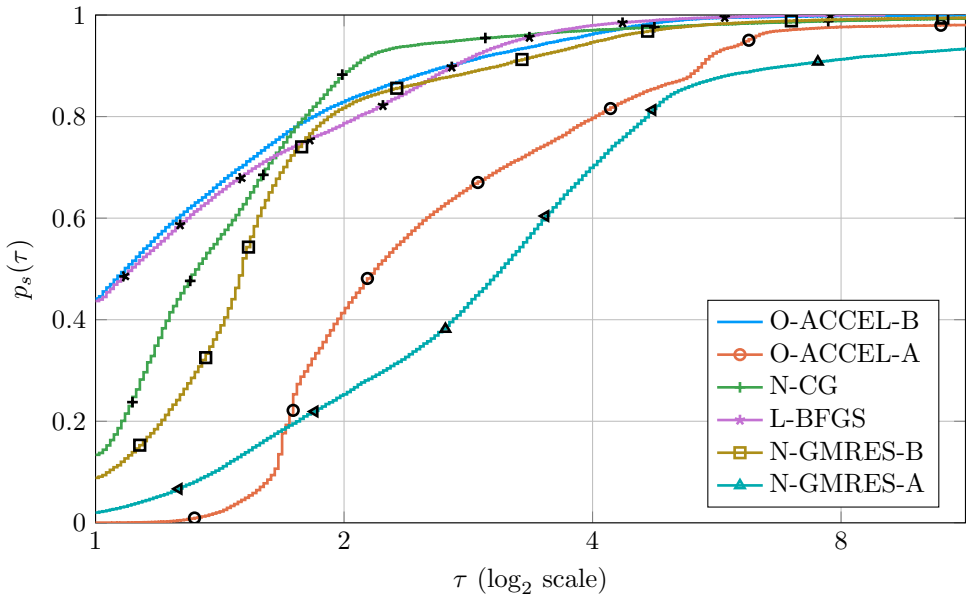}
  \caption{Performance profiles, defined in~\eqref{eq:perf_profile},
    for Problems A--G.
    O-ACCEL preconditioned with a fixed-step steepest descent (B) and
    L-BFGS mostly
    outperform the rest, except for higher factors of $\tau$. They are
    also more robust, solving the largest proportion
    of the problems when the computational budget is large.
  }\label{fig:perf_prof_all}
\end{figure}
\Cref{fig:perf_prof_all} plots the performance profile of the $n_s=6$
solvers considered: N-CG, L-BFGS, and N-GMRES and O-ACCEL with steepest descent
preconditioning using both a line search (A) and a fixed step size (B).
It is clear that O-ACCEL-B and L-BFGS are the  best performers across the test set.
For \SI{44}{\percent} of the test problems they reach tolerance in the fewest
$\Obj$ evaluations, and they also solve the largest proportion of problems
within higher factors $\tau$ of the best performance ratio.
There is also a region where N-CG does particularly well, solving the
largest proportion of problems within two to three times the highest
performing solver.
The worst performers are N-GMRES-A and O-ACCEL-A, mainly due to
the high amount of work that the line search must do to satisfy the
Wolfe conditions along the steepest descent directions.

It is notable that O-ACCEL-B is competitive with L-BFGS on the test
set. Tests, not presented in this work, indicate that the L-BFGS performance
improves by using a line search with
Wolfe curvature condition parameter
$c_2=0.9$, rather than $c_2=0.1$ as used in this manuscript.
The main focus of this manuscript is, however, to investigate the
potential improvement of minimizing the objective rather than
an $\ell_2$ norm of the gradient. Thus, we are more interested in the
comparison between N-GMRES and O-ACCEL.
The two plots in \Cref{fig:perf_prof_nls_nsd} show the performance profiles
comparing N-GMRES and O-ACCEL, and in both cases show a significant
improvement by minimizing the objective. In fact, O-ACCEL reaches
tolerance first on \SIrange{63}{71}{\percent} of the test problems.
The instances where N-GMRES does better is primarily in
Problems E, F, and G, as can be seen from \Cref{tbl:probEG} in the Appendix.
One of the findings of \citet{sterck2013steepest} was that N-GMRES
with line search-steepest descent often stagnated or converged very
slowly. From the left plot of \Cref{fig:perf_prof_nls_nsd}, we see
that this
issue is reduced with the O-ACCEL acceleration. It also
turns out that O-ACCEL-A has a larger success rate over the test
set than N-GMRES-A.
\begin{figure}[htb]
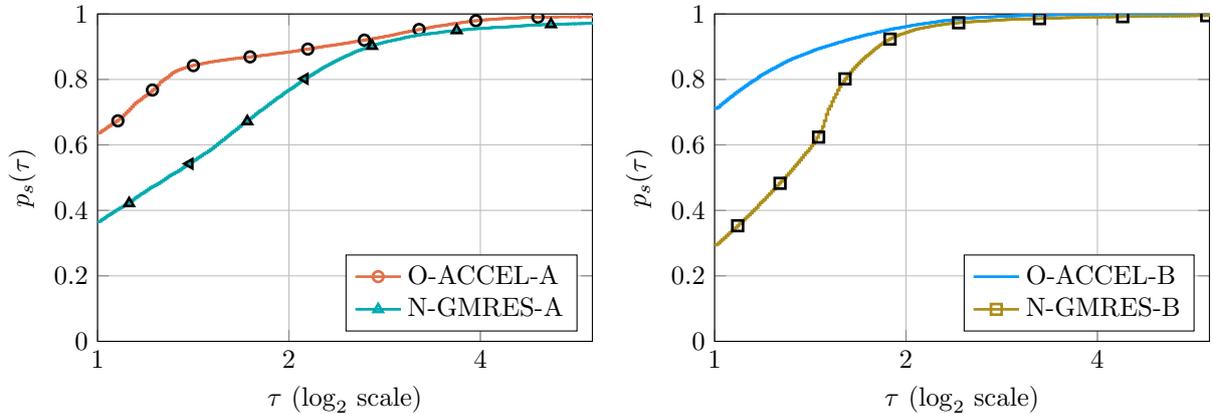

  \centering
  \begin{minipage}{0.49\textwidth}
    \includegraphics[width=\textwidth,axisratio=1.5]{perf_prof_nls}
  \end{minipage}
  \hfill
  \begin{minipage}{0.49\textwidth}
    \includegraphics[width=\textwidth,axisratio=1.5]{perf_prof_nsd}
  \end{minipage}
  \caption{Performance profiles comparing N-GMRES and O-ACCEL with
    steepest descent with line search (A, left) and without (B, right).
    O-ACCEL outperforms N-GMRES in both cases on our test set.
    Note that the lines of O-ACCEL-A and N-GMRES-A cross in
    \Cref{fig:perf_prof_all}, but not in the left figure here, because
    the performance profiles change depending on the set of solvers considered.
  }\label{fig:perf_prof_nls_nsd}
\end{figure}

\subsection{The tensor optimization problem from \citeauthor{sterck2013steepest}}\label{subsec:tensor_optim}
The original motivation for N-GMRES was to improve convergence for a tensor
optimization problem~\cite{sterck2012nonlinear}.
\citet{sterck2012nonlinear} and \citet{sterck2013steepest} show that
using N-GMRES with a  domain-specific ALS
preconditioner is better than generic optimizers such as L-BFGS
and N-CG.
\citet{sterck2013steepest} states that
``In this problem, a rank-three canonical tensor approximation (with 450
variables) is sought for a three-way data tensor of size $50\times
50\times 50$. The data tensor is generated starting from a canonical tensor
with specified rank and random factor matrices that are modified to
have prespecified column colinearity, and noise is added. This is a
standard canonical tensor decomposition test problem
\citep{acar2011scalable}.''
For this manuscript, we run the \num{1000} realisations of the
test problem using the code provided by \citet{sterck2013steepest}
with the parameter values described in
\Cref{subsec:experiment_design}.
The algorithms tested for this problem are vanilla ALS, N-GMRES-ALS,
O-ACCEL-ALS, N-CG, and L-BFGS.
\Cref{fig:perf_prof_tensor_cp_all} and
\Cref{tbl:quantiles_tensor_cp_all} show the performance profiles and
quantiles for the number of $\Obj$ evaluations required to reach
tolerance.
We see that O-ACCEL-ALS and N-GMRES-ALS perform  better than the
other algorithms, which underscores the advantage of applying these
acceleration methods to domain-specific algorithms.
\begin{figure}[htbp]
  \centering
  \begin{subfigure}[b]{0.5\textwidth}
    \includegraphics[width=\textwidth,axisratio=1.2]{perf_prof_tensor_cp_all}
    \caption{Performance profiles}\label{fig:perf_prof_tensor_cp_all}
  \end{subfigure}%
  \begin{subfigure}[b]{0.5\textwidth}
    \centering
    \begin{tabular}{lS[table-auto-round,table-format=4]
      S[table-auto-round,table-format=4]
      S[table-auto-round,table-format=4]}
      \toprule
      Algorithm&{$Q_{0.1}$}&{$Q_{0.5}$}&{$Q_{0.9}$}\\
      \midrule
      ALS&877.5&1107.0&1212.0\\
      \rowcolor[gray]{0.9}
      O-ACCEL-ALS&170.0&227.0&297.0\\
      N-CG&405.0&806.5&2449.5\\
      L-BFGS&301.5&437.5&3036.5\\
      N-GMRES-ALS&169.0&234.0&311.5\\
      \bottomrule
    \end{tabular}
    \vspace{5em}
    \caption{$\Obj$ evaluations}\label{tbl:quantiles_tensor_cp_all}
  \end{subfigure}
  \caption{Numerical results from the tensor optimization test
    problem. O-ACCEL and N-GMRES perform significantly better than the
    other solvers.}
\end{figure}

\subsection{CUTEst test problems}
The test problems we have considered so far were
taken from \citet{sterck2013steepest} and originally used to promote
N-GMRES. We
finish by presenting results from a numerical experiment using
problems from the CUTEst problem set~\cite{gould2015cutest}.  For this
experiment, we compare the solvers O-ACCEL-B, L-BFGS, and N-GMRES-B,
with the parameter values described in
\Cref{subsec:experiment_design}.
The minima are not known for many of the CUTEst problems,
and so we change the tolerance criterion to be defined
in terms of the relative decrease of the gradient norm.
The performance measure used for this experiment is
\begin{equation}\label{eq:def_performance_measure_cutest}
  t_{p,s} = \min\{k\geq 1 \mid \|\Grad(x^{(k)})\|_\infty \leq 10^{-8}\|\Grad(x^{(0)})\|_\infty\}.
\end{equation}
A solver run is labelled as failed if it does not reach tolerance
within \num{2000} iterations.

We run the experiment using
implementations of the solvers from the package Optim
\citep{mogensen2018optim} of the Julia programming language
\citep{bezanson2017julia}. To be sure, we have also verified that the
Optim code yields the same results as the MATLAB code for Problems
A--G.

\begin{figure}[htb]
  \centering
  \begin{minipage}{0.49\textwidth}
    \includegraphics[width=\textwidth,axisratio=1.5]{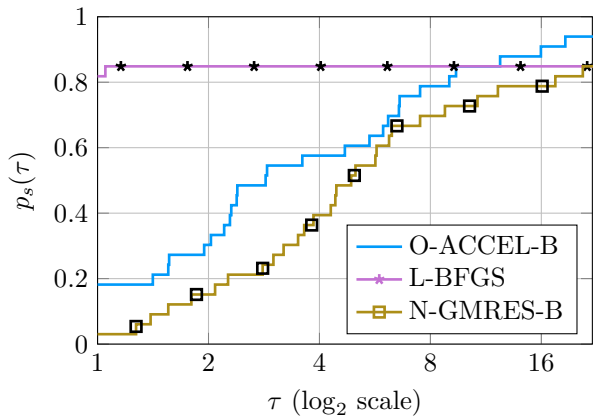}
  \end{minipage}
  \hfill
  \begin{minipage}{0.49\textwidth}
    \includegraphics[width=\textwidth,axisratio=1.5]{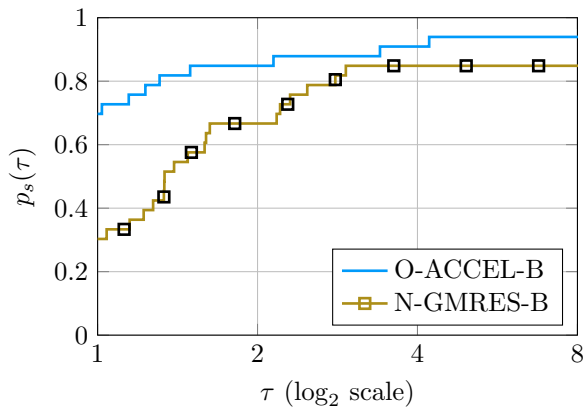}
  \end{minipage}
  \caption{Performance profiles for the CUTEst test problems from
    \Cref{tbl:cutest_table1,tbl:cutest_table2}.
  L-BFGS is the highest performing most of the time, however,
  O-ACCEL-B reaches tolerance for more problems.}\label{fig:perf_prof_cutest}
\end{figure}
The \num{33} problems we consider are listed in
\Cref{tbl:cutest_table1,tbl:cutest_table2} of the appendix together
with the results of the numerical experiment.
We selected the problems with dimension $n=50$ to \num{10000} that
satisfy the two criteria
(i) the objective type is in the category ``other'' (ii) at least one
of the solvers succeed in reaching tolerance.
\Cref{fig:perf_prof_cutest} shows performance profiles from the
experiment. L-BFGS reaches tolerance first for most of the
problems, however, O-ACCEL reaches tolerance within
\num{2000} iterations for more of the test problems. In the problems
where L-BFGS does not reach tolerance it stops because it fails
prematurely, whilst N-GMRES-B only fails due to reaching \num{2000} iterations.
We believe the poorer performance of the acceleration
algorithms for the CUTEst problems, compared to the previous
experiments, is due to the poor performance of the steepest descent
preconditioner on these problems.
Again, O-ACCEL-B performs better than N-GMRES-B, which underscores our claim
that accelerating based on the objective function is better than
accelerating based on the gradient norm.

\section{Conclusion}\label{sec:conclusion}
We have proposed a simple acceleration algorithm
for optimization, based on the nonlinear GMRES (N-GMRES) algorithm by
\citet{washio1997krylov,sterck2013steepest}.
N-GMRES for optimization aims to accelerate a solver
step when solving the nonlinear system $\nabla\Obj(x)=0$ by minimizing the
residual in the $\ell_2$ norm over a subspace from previous iterates.
The acceleration step consists of solving a small linear system that
arises from a linearization of the gradient.

We propose to take advantage of the
structure of the optimization problem and instead accelerate based on the
objective value $\Obj(x)$.
This new approach, labelled O-ACCEL, shows a significant improvement to the original
N-GMRES algorithm in numerical tests when accelerating a steepest descent
solver. The first test problems are taken from
\citet{sterck2013steepest} and run under the same conditions that
proved to be beneficial for N-GMRES.
Further tests on a selection of CUTEst problems strengthen the
conclusion that O-ACCEL outperforms N-GMRES.
Another strength of these acceleration algorithms is that they can be
combined with many types of optimizers. We have seen O-ACCEL's efficiency
with steepest descent, and accelerating quasi-Newton, Newton
methods, and domain-specific methods have potential to reduce costs
for more expensive algorithms.
For example, in \citet{sterck2013steepest} it is shown that N-GMRES significantly
accelerates the alternating least squares algorithm (ALS), which already without
acceleration performs much better than L-BFGS and N-CG on a standard
canonical tensor decomposition problem. Our numerical tests show that
O-ACCEL further improves the ALS convergence for this problem.

There are two particular paths of interest to improve the proposed acceleration
scheme. The
first is to reduce the cost by not using a line search between the proposed steps
by the solver and O-ACCEL. One can instead rely on heuristics along the lines of those
proposed by \citet{washio1997krylov}.
The second is to find better heuristics for choosing previous
iterates to use in the acceleration step. Currently, no choices
are made, other than discarding all iterates when problems appear.
Better guidelines for the number of previous iterates to store is
another topic of interest, especially when memory storage is limited.

We would like to investigate connections between the
proposed O-ACCEL acceleration step and other optimization
procedures, in the same fashion that \citet{fang2009two} put Anderson
acceleration in the context of a family of Broyden-type approximations
of the inverse Jacobian (Hessian).
The preliminary analysis presented in this manuscript shows that, for convex
quadratic objectives, O-ACCEL with a gradient descent preconditioner
is equivalent to FOM for linear systems. As FOM is equivalent to
CG for symmetric positive definite systems, we can
view O-ACCEL in the context of N-CG methods using a larger
history size than usual.
There are many new ideas for improving step directions
based on previous iterates, such as the acceleration scheme by
\citet{damien2016regularized}, and Block BFGS by \citet{gao2016block}.
A better understanding of the overlaps between these and more classical
optimization procedures can provide useful guidance for further research.

Further work is needed to test O-ACCEL on a wider range of problems,
with comparisons to other state-of-the-art implementations of solvers and
accelerators, in order to provide guidance as to when a method is appropriate.
For example, on Problems A--G, O-ACCEL accelerating steepest descent is
superior to N-CG and slightly better than L-BFGS. These results may, however, be
due to implementations from \citet{sterck2013steepest} and test problems favouring the
acceleration algorithms. They are still indicative of the power of
objective value based optimization, a research track that is worth
pursuing further.

\paragraph*{Acknowledgements.}
The author would like to thank Coralia Cartis for her suggestions on
how to present the results in this manuscript, and his supervisors
Jeff Dewynne and Chris Farmer for their input.
Antoine Levitt suggested to look into the connection between O-ACCEL
and FOM.
This work originally
came about from an investigation of N-GMRES for nonlinear solvers for
PDEs with Patrick Farrell~\citep{riseth2015nonlinear}.

\paragraph*{Funding.}
This publication is based on work partially supported by the EPSRC
Centre For Doctoral Training in Industrially Focused Mathematical
Modelling (EP/L015803/1).

\paragraph*{Data access.}
The code used in producing this manuscript is available at
\url{https://github.com/anriseth/objective_accel_code}.
It includes MATLAB implementations of the N-GMRES and O-ACCEL
algorithms, and code to run Problems A--G as well as the tensor
optimization problem.
Implementations of N-GMRES and O-ACCEL are also available in
the optimization package Optim \citep{mogensen2018optim} of the Julia
programming language \citep{bezanson2017julia}.

\appendix

\section*{Appendix: Tables of numerical results}
To supplement the performance profiles in the manuscript, we include
tables that present statistics of the solver performances for the
individual test problems.

\Cref{tbl:probAD,tbl:probEG} show the results from test problems A--G.
Each of the problems were tested with
different sizes $n$, and for each value $n$ the problems were run
\num{1000} times in order to create statistics.
The tables report the \num{0.1}, \num{0.5}, and \num{0.9} quantiles of
$\Obj$ evaluations to reach the objective value
reduction in~\eqref{eq:def_performance_measure}, denoted by $Q_{0.1}$,
$Q_{0.5}$, and $Q_{0.9}$ respectively.
\Cref{tbl:probAD} provides results for problems A--D, and
\Cref{tbl:probEG} for the remaining problems E--G.

\Cref{tbl:cutest_table1,tbl:cutest_table2} show the results from the
CUTEst problems, where ``Fail'' means failure to reach the gradient
value reduction in~\eqref{eq:def_performance_measure_cutest} within
\num{2000} iterations.  The norm used for the gradient values in the
tables is the infinity norm.

\begin{table}[p]
  \sisetup{detect-weight=true,detect-inline-weight=math}
  \centering
  \begin{tabular}{l*{2}{
    S[table-auto-round,table-format=4]
    S[table-auto-round,table-format=4]
    S[table-auto-round,table-format=4]}
    }
    \toprule
    Algorithm
    &\multicolumn{3}{c}{A, $n=100$}
    &\multicolumn{3}{c}{A, $n=200$}\\
    \cmidrule(lr){2-4}\cmidrule(lr){5-7}
    &{$Q_{0.1}$}&{$Q_{0.5}$}&{$Q_{0.9}$}
                &{$Q_{0.1}$}&{$Q_{0.5}$}&{$Q_{0.9}$}\\
    \midrule
    \rowcolor[gray]{0.9}
    O-ACCEL-B&75.0&79.0&81.0&103.0&107.0&111.0\\
    O-ACCEL-A&131.0&136.0&140.0&171.0&179.0&184.0\\
    N-CG&87.0&93.0&99.0&113.0&131.0&145.0\\
    \rowcolor[gray]{0.9}
    L-BFGS&75.0&79.0&81.0&103.0&107.0&111.0\\
    N-GMRES-B&111.0&117.0&122.0&158.0&169.0&192.0\\
    N-GMRES-A&166.0&246.0&335.5&306.5&414.0&510.0\\
    \bottomrule
  \end{tabular}
  \\[0.5em]
  \begin{tabular}{l*{2}{
    S[table-auto-round,table-format=4]
    S[table-auto-round,table-format=4]
    S[table-auto-round,table-format=4]}
    }
    \toprule
    Algorithm
    &\multicolumn{3}{c}{B, $n=100$}
    &\multicolumn{3}{c}{B, $n=200$}\\
    \cmidrule(lr){2-4}\cmidrule(lr){5-7}
    &{$Q_{0.1}$}&{$Q_{0.5}$}&{$Q_{0.9}$}
                &{$Q_{0.1}$}&{$Q_{0.5}$}&{$Q_{0.9}$}\\
    \midrule
    O-ACCEL-B&183.0&267.0&415.5&262.0&364.5&595.0\\
    O-ACCEL-A&258.0&389.0&545.5&377.0&478.0&799.5\\
    N-CG&134.0&211.0&560.0&221.0&359.0&1598.0\\
    \rowcolor[gray]{0.9}
    L-BFGS&76.0&100.0&169.0&99.0&127.0&292.0\\
    N-GMRES-B&215.0&314.5&541.5&317.0&433.0&839.5\\
    N-GMRES-A&272.0&648.0&1515.5&452.0&809.0&2203.5\\
    \bottomrule
  \end{tabular}
  \\[0.5em]
  \begin{tabular}{l*{2}{
    S[table-auto-round,table-format=4]
    S[table-auto-round,table-format=4]
    S[table-auto-round,table-format=4]}
    }
    \toprule
    Algorithm
    &\multicolumn{3}{c}{C, $n=100$}
    &\multicolumn{3}{c}{C, $n=200$}\\
    \cmidrule(lr){2-4}\cmidrule(lr){5-7}
    &{$Q_{0.1}$}&{$Q_{0.5}$}&{$Q_{0.9}$}
                &{$Q_{0.1}$}&{$Q_{0.5}$}&{$Q_{0.9}$}\\
    \midrule
    O-ACCEL-B&112.5&136.0&177.5&151.0&176.0&214.5\\
    O-ACCEL-A&187.5&208.0&258.5&264.0&292.0&324.0\\
    N-CG&165.0&187.0&215.0&259.0&298.0&344.0\\
    \rowcolor[gray]{0.9}
    L-BFGS&104.0&114.0&125.0&147.5&160.0&177.0\\
    N-GMRES-B&142.0&164.0&208.0&219.0&253.5&304.0\\
    N-GMRES-A&264.0&333.0&459.0&508.0&620.0&854.0\\
    \bottomrule
  \end{tabular}
  \\[0.5em]
  \begin{tabular}{l*{2}{
    S[table-auto-round,table-format=4]
    S[table-auto-round,table-format=4]
    S[table-auto-round,table-format=4]}
    }
    \toprule
    Algorithm
    &\multicolumn{3}{c}{D, $n=500$}
    &\multicolumn{3}{c}{D, $n=1000$}\\
    \cmidrule(lr){2-4}\cmidrule(lr){5-7}
    &{$Q_{0.1}$}&{$Q_{0.5}$}&{$Q_{0.9}$}
                &{$Q_{0.1}$}&{$Q_{0.5}$}&{$Q_{0.9}$}\\
    \midrule
    \rowcolor[gray]{0.9}
    O-ACCEL-B&93.0&105.0&123.0&91.0&98.0&116.0\\
    O-ACCEL-A&193.0&233.0&276.5&192.0&233.0&280.0\\
    N-CG&158.0&188.0&196.0&162.0&190.0&197.0\\
    L-BFGS&128.0&155.0&194.0&128.5&153.0&188.5\\
    N-GMRES-B&141.0&163.0&193.0&142.0&167.0&193.0\\
    N-GMRES-A&284.0&349.0&508.0&290.0&349.0&470.5\\
    \midrule
    Algorithm
    &\multicolumn{3}{c}{D, $n=50000$}
    &\multicolumn{3}{c}{D, $n=100000$}\\
    \cmidrule(lr){2-4}\cmidrule(lr){5-7}
    &{$Q_{0.1}$}&{$Q_{0.5}$}&{$Q_{0.9}$}
                &{$Q_{0.1}$}&{$Q_{0.5}$}&{$Q_{0.9}$}\\
    \midrule
    \rowcolor[gray]{0.9}
    O-ACCEL-B&101.0&117.0&132.0&122.0&126.0&135.0\\
    O-ACCEL-A&195.0&226.0&276.0&196.0&225.0&271.0\\
    N-CG&159.0&187.0&196.0&159.0&188.0&196.0\\
    L-BFGS&131.0&156.0&190.0&130.0&156.0&191.0\\
    N-GMRES-B&154.0&178.0&215.0&162.5&190.0&230.5\\
    N-GMRES-A&312.0&378.0&525.0&319.5&394.0&561.5\\
    \bottomrule
  \end{tabular}
  \caption{Quantiles reporting $\Obj$ evaluations to reach
    tolerance for each solver. Grey rows highlight the solver with the
    best \num{0.5} qauntile.
    L-BFGS performs best for the easier problems, whilst O-ACCEL
    handles the difficult problems better. In Problem A, the L-BFGS and
    O-ACCEL performance measures are so similar that the quantiles are
    the same.}\label{tbl:probAD}
\end{table}

\begin{table}[p]
  \centering
  \begin{tabular}{l*{2}{
    S[table-auto-round,table-format=4]
    S[table-auto-round,table-format=4]
    S[table-auto-round,table-format=4]}
    }
    \toprule
    Algorithm
    &\multicolumn{3}{c}{E, $n=100$}
    &\multicolumn{3}{c}{E, $n=200$}\\
    \cmidrule(lr){2-4}\cmidrule(lr){5-7}
    &{$Q_{0.1}$}&{$Q_{0.5}$}&{$Q_{0.9}$}
                &{$Q_{0.1}$}&{$Q_{0.5}$}&{$Q_{0.9}$}\\
    \midrule
    \rowcolor[gray]{0.9}
    O-ACCEL-B&190.0&222.0&265.0&198.0&228.0&273.5\\
    O-ACCEL-A&301.0&349.0&624.5&312.0&371.0&780.5\\
    N-CG&204.5&238.0&283.0&213.0&245.0&290.0\\
    L-BFGS&463.0&626.5&964.5&479.5&638.5&1035.5\\
    N-GMRES-B&231.5&267.0&330.0&235.0&268.0&337.5\\
    N-GMRES-A&280.0&332.0&395.0&284.0&335.0&401.0\\
    \midrule
    Algorithm
    &\multicolumn{3}{c}{E, $n=50000$}
    &\multicolumn{3}{c}{E, $n=100000$}\\
    \cmidrule(lr){2-4}\cmidrule(lr){5-7}
    &{$Q_{0.1}$}&{$Q_{0.5}$}&{$Q_{0.9}$}
                &{$Q_{0.1}$}&{$Q_{0.5}$}&{$Q_{0.9}$}\\
    \midrule
    O-ACCEL-B&368.0&487.0&689.0&399.5&536.0&798.0\\
    O-ACCEL-A&744.5&1157.0&1355.5&763.5&1207.0&1417.0\\
    N-CG&297.0&360.0&461.0&321.0&384.0&490.5\\
    L-BFGS&599.0&703.0&851.5&625.5&725.0&879.0\\
    \rowcolor[gray]{0.9}
    N-GMRES-B&275.0&335.0&738.0&258.0&318.0&848.0\\
    N-GMRES-A&309.5&390.5&561.5&318.0&402.0&609.0\\
    \bottomrule
  \end{tabular}
  \\[0.5em]
  \begin{tabular}{l*{2}{
    S[table-auto-round,table-format=4]
    S[table-auto-round,table-format=4]
    S[table-auto-round,table-format=4]}
    }
    \toprule
    Algorithm
    &\multicolumn{3}{c}{F, $n=200$}
    &\multicolumn{3}{c}{F, $n=500$}\\
    \cmidrule(lr){2-4}\cmidrule(lr){5-7}
    &{$Q_{0.1}$}&{$Q_{0.5}$}&{$Q_{0.9}$}
                &{$Q_{0.1}$}&{$Q_{0.5}$}&{$Q_{0.9}$}\\
    \midrule
    O-ACCEL-B&53.0&71.0&118.0&44.0&55.0&96.5\\
    O-ACCEL-A&81.0&93.0&110.0&84.0&102.0&121.0\\
    \rowcolor[gray]{0.9}
    N-CG&34.0&46.0&60.0&33.0&47.0&69.0\\
    \rowcolor[gray]{0.9}
    L-BFGS&41.0&48.0&56.0&34.0&44.0&51.0\\
    N-GMRES-B&48.0&59.0&110.0&43.0&51.0&88.5\\
    N-GMRES-A&76.0&87.0&99.0&78.0&92.0&107.0\\
    \bottomrule
  \end{tabular}
  \\[0.5em]
  \begin{tabular}{l*{2}{
    S[table-auto-round,table-format=4]
    S[table-auto-round,table-format=4]
    S[table-auto-round,table-format=4]}
    }
    \toprule
    Algorithm
    &\multicolumn{3}{c}{G, $n=100$}
    &\multicolumn{3}{c}{G, $n=200$}\\
    \cmidrule(lr){2-4}\cmidrule(lr){5-7}
    &{$Q_{0.1}$}&{$Q_{0.5}$}&{$Q_{0.9}$}
                &{$Q_{0.1}$}&{$Q_{0.5}$}&{$Q_{0.9}$}\\
    \midrule
    O-ACCEL-B&148.0&211.5&296.0&195.5&224.0&257.5\\
    O-ACCEL-A&301.5&940.0&1078.0&220.0&815.0&956.5\\
    N-CG&76.0&191.0&201.0&53.0&165.0&174.0\\
    \rowcolor[gray]{0.9}
    L-BFGS&66.0&173.0&180.0&53.0&150.0&156.0\\
    N-GMRES-B&161.0&216.0&266.0&166.5&210.0&245.0\\
    N-GMRES-A&528.0&764.0&4518.0&203.0&720.0&4526.0\\
    \bottomrule
  \end{tabular}
  \caption{Quantiles reporting $\Obj$ evaluations to reach
    tolerance for each solver. Grey rows highlight the solver with the
    best \num{0.5} qauntile.
    N-GMRES performs best at the median range for two problems,
    however it is less robust as can be seen from the upper quantile.
  }\label{tbl:probEG}
\end{table}

\begin{table}[p]
  \small
  \centering
  \sisetup{tight-spacing=true,table-auto-round}
\begin{tabular}{*{3}{l}*{2}{S[table-format=4]}
*{2}{S[table-format=-1.1e-2]}c}
  \toprule
  \multicolumn{2}{c}{Problem}
  & Solver & {Iter} & {$\Obj$-calls} & {$\Obj_{\textnormal{min}}$} &
  {$\|\Grad_{\textnormal{min}}\|$} & Fail \\
  \midrule
\multirow{3}{*}{ARWHEAD} & $ n = \num{5000} $  & O-ACCEL-B & 9 & 20 & 0.00e+00 & 2.33e-06 &  \\
 & $\Obj_0 = \num{1.5e+04} $  & L-BFGS & 6 & 21 & 0.00e+00 & 4.82e-07 &  \\
 & ${\|\Grad_0\|} = \num{4.0e+04} $  & N-GMRES-B & 6 & 20 & 2.63e-09 & 2.14e-05 &  \\
\addlinespace
\multirow{3}{*}{BOX} & $ n = \num{10000} $  & O-ACCEL-B & 19 & 86 & -1.86e+03 & 1.95e-10 &  \\
 & $\Obj_0 = \num{0.0e+00} $  & L-BFGS & 7 & 14 & -1.86e+03 & 4.00e-09 &  \\
 & ${\|\Grad_0\|} = \num{5.0e-01} $  & N-GMRES-B & 30 & 105 & -1.86e+03 & 2.19e-09 &  \\
\addlinespace
\multirow{3}{*}{COSINE} & $ n = \num{10000} $  & O-ACCEL-B & 52 & 272 & -1.00e+04 & 7.14e-09 &  \\
 & $\Obj_0 = \num{8.8e+03} $  & L-BFGS & 12 & 22 & -1.00e+04 & 3.10e-09 &  \\
 & ${\|\Grad_0\|} = \num{9.6e-01} $  & N-GMRES-B & 29 & 80 & -1.00e+04 & 1.96e-09 &  \\
\addlinespace
\multirow{3}{*}{CRAGGLVY} & $ n = \num{5000} $  & O-ACCEL-B & 157 & 478 & 1.69e+03 & 4.71e-05 &  \\
 & $\Obj_0 = \num{2.7e+06} $  & L-BFGS & 57 & 133 & 1.69e+03 & 4.01e-05 &  \\
 & ${\|\Grad_0\|} = \num{5.6e+03} $  & N-GMRES-B & 229 & 760 & 1.69e+03 & 3.09e-05 &  \\
\addlinespace
\multirow{3}{*}{DIXMAANA} & $ n = \num{3000} $  & O-ACCEL-B & 48 & 223 & 1.00e+00 & 6.11e-10 &  \\
 & $\Obj_0 = \num{2.9e+04} $  & L-BFGS & 6 & 12 & 1.00e+00 & 9.90e-16 &  \\
 & ${\|\Grad_0\|} = \num{2.8e+01} $  & N-GMRES-B & 13 & 53 & 1.00e+00 & 1.47e-10 &  \\
\addlinespace
\multirow{3}{*}{DIXMAANB} & $ n = \num{3000} $  & O-ACCEL-B & 31 & 99 & 1.00e+00 & 3.16e-08 &  \\
 & $\Obj_0 = \num{4.7e+04} $  & L-BFGS & 6 & 11 & 1.00e+00 & 2.25e-07 &  \\
 & ${\|\Grad_0\|} = \num{4.0e+01} $  & N-GMRES-B & 35 & 228 & 1.00e+00 & 3.79e-10 &  \\
\addlinespace
\multirow{3}{*}{DIXMAANC} & $ n = \num{3000} $  & O-ACCEL-B & 30 & 105 & 1.00e+00 & 4.05e-08 &  \\
 & $\Obj_0 = \num{8.2e+04} $  & L-BFGS & 7 & 14 & 1.00e+00 & 2.27e-08 &  \\
 & ${\|\Grad_0\|} = \num{7.6e+01} $  & N-GMRES-B & 13 & 49 & 1.00e+00 & 4.40e-08 &  \\
\addlinespace
\multirow{3}{*}{DIXMAAND} & $ n = \num{3000} $  & O-ACCEL-B & 23 & 75 & 1.00e+00 & 1.47e-06 &  \\
 & $\Obj_0 = \num{1.6e+05} $  & L-BFGS & 8 & 16 & 1.00e+00 & 2.55e-07 &  \\
 & ${\|\Grad_0\|} = \num{1.5e+02} $  & N-GMRES-B & 16 & 78 & 1.00e+00 & 4.89e-07 &  \\
\addlinespace
\multirow{3}{*}{DIXMAANE} & $ n = \num{3000} $  & O-ACCEL-B & 394 & 1109 & 1.00e+00 & 2.61e-07 &  \\
 & $\Obj_0 = \num{2.2e+04} $  & L-BFGS & 241 & 485 & 1.00e+00 & 2.67e-07 &  \\
 & ${\|\Grad_0\|} = \num{2.7e+01} $  & N-GMRES-B & 885 & 2751 & 1.00e+00 & 2.60e-07 &  \\
\addlinespace
\multirow{3}{*}{DIXMAANF} & $ n = \num{3000} $  & O-ACCEL-B & 282 & 765 & 1.00e+00 & 3.47e-07 &  \\
 & $\Obj_0 = \num{4.1e+04} $  & L-BFGS & 195 & 393 & 1.00e+00 & 3.73e-07 &  \\
 & ${\|\Grad_0\|} = \num{3.9e+01} $  & N-GMRES-B & 567 & 1687 & 1.00e+00 & 3.87e-07 &  \\
\addlinespace
\multirow{3}{*}{DIXMAANG} & $ n = \num{3000} $  & O-ACCEL-B & 277 & 770 & 1.00e+00 & 6.93e-07 &  \\
 & $\Obj_0 = \num{7.6e+04} $  & L-BFGS & 165 & 334 & 1.00e+00 & 6.88e-07 &  \\
 & ${\|\Grad_0\|} = \num{7.5e+01} $  & N-GMRES-B & 570 & 1673 & 1.00e+00 & 7.16e-07 &  \\
\addlinespace
\multirow{3}{*}{DIXMAANH} & $ n = \num{3000} $  & O-ACCEL-B & 236 & 655 & 1.00e+00 & 1.51e-06 &  \\
 & $\Obj_0 = \num{1.5e+05} $  & L-BFGS & 146 & 297 & 1.00e+00 & 1.49e-06 &  \\
 & ${\|\Grad_0\|} = \num{1.5e+02} $  & N-GMRES-B & 620 & 1835 & 1.00e+00 & 1.38e-06 &  \\
\addlinespace
\multirow{3}{*}{DIXMAANK} & $ n = \num{3000} $  & O-ACCEL-B & 862 & 1877 & 1.00e+00 & 7.09e-07 &  \\
 & $\Obj_0 = \num{7.4e+04} $  & L-BFGS & 392 & 787 & 1.00e+00 & 7.06e-07 &  \\
 & ${\|\Grad_0\|} = \num{7.4e+01} $  & N-GMRES-B & 556 & 1640 & 1.00e+00 & 7.17e-07 &  \\
\addlinespace
\multirow{3}{*}{DIXMAANL} & $ n = \num{3000} $  & O-ACCEL-B & 267 & 685 & 1.00e+00 & 1.44e-06 &  \\
 & $\Obj_0 = \num{1.5e+05} $  & L-BFGS & 240 & 485 & 1.00e+00 & 1.47e-06 &  \\
 & ${\|\Grad_0\|} = \num{1.5e+02} $  & N-GMRES-B & 378 & 1096 & 1.00e+00 & 1.35e-06 &  \\
\addlinespace
\multirow{3}{*}{DIXMAANP} & $ n = \num{3000} $  & O-ACCEL-B & 1426 & 3149 & 1.00e+00 & 1.25e-06 &  \\
 & $\Obj_0 = \num{7.1e+04} $  & L-BFGS & 549 & 1102 & 1.00e+00 & 1.25e-06 &  \\
 & ${\|\Grad_0\|} = \num{1.3e+02} $  & N-GMRES-B & 1037 & 3091 & 1.00e+00 & 1.11e-06 &  \\
\addlinespace
\multirow{3}{*}{EDENSCH} & $ n = \num{2000} $  & O-ACCEL-B & 75 & 246 & 1.20e+04 & 1.83e-05 &  \\
 & $\Obj_0 = \num{7.4e+06} $  & L-BFGS & 21 & 45 & 1.20e+04 & 1.96e-05 &  \\
 & ${\|\Grad_0\|} = \num{2.2e+03} $  & N-GMRES-B & 54 & 200 & 1.20e+04 & 7.11e-06 &  \\
\bottomrule
\end{tabular}

  \caption{Results from the CUTEst problems.}\label{tbl:cutest_table1}
\end{table}

\begin{table}[p]
  \small
  \centering
  \sisetup{tight-spacing=true,table-auto-round}
\begin{tabular}{*{3}{l}*{2}{S[table-format=4]}
*{2}{S[table-format=-1.1e-2]}c}
  \toprule
  \multicolumn{2}{c}{Problem}
  & Solver & {Iter} & {$\Obj$-calls} & {$\Obj_{\textnormal{min}}$} &
  {$\|\Grad_{\textnormal{min}}\|$} & Fail \\
  \midrule
\multirow{3}{*}{EG2} & $ n = \num{1000} $  & O-ACCEL-B & 6 & 14 & -9.99e+02 & 1.68e-06 &  \\
 & $\Obj_0 = \num{-8.4e+02} $  & L-BFGS & 3 & 9 & -9.99e+02 & 4.06e-07 &  \\
 & ${\|\Grad_0\|} = \num{5.4e+02} $  & N-GMRES-B & 6 & 14 & -9.99e+02 & 3.64e-06 &  \\
\addlinespace
\multirow{3}{*}{ENGVAL1} & $ n = \num{5000} $  & O-ACCEL-B & 47 & 250 & 5.55e+03 & 8.52e-07 &  \\
 & $\Obj_0 = \num{2.9e+05} $  & L-BFGS & 33 & 366 & 5.55e+03 & 2.75e-06 & $\times$ \\
 & ${\|\Grad_0\|} = \num{1.2e+02} $  & N-GMRES-B & 65 & 318 & 5.55e+03 & 7.60e-07 &  \\
\addlinespace
\multirow{3}{*}{FLETBV3M} & $ n = \num{5000} $  & O-ACCEL-B & 143 & 989 & -2.48e+05 & 4.83e-09 &  \\
 & $\Obj_0 = \num{2.0e+02} $  & L-BFGS & 24 & 62 & -2.49e+05 & 5.12e-10 &  \\
 & ${\|\Grad_0\|} = \num{7.1e-01} $  & N-GMRES-B & 92 & 756 & -2.49e+05 & 5.02e-09 &  \\
\addlinespace
\multirow{3}{*}{FMINSRF2} & $ n = \num{5625} $  & O-ACCEL-B & 1510 & 4384 & 1.00e+00 & 9.58e-09 & $\times$ \\
 & $\Obj_0 = \num{2.8e+01} $  & L-BFGS & 1537 & 3367 & 1.00e+00 & 2.23e-10 &  \\
 & ${\|\Grad_0\|} = \num{2.4e-02} $  & N-GMRES-B & 2000 & 4661 & 1.49e+00 & 9.13e-03 & $\times$ \\
\addlinespace
\multirow{3}{*}{FMINSURF} & $ n = \num{5625} $  & O-ACCEL-B & 1741 & 5007 & 1.00e+00 & 9.72e-09 & $\times$ \\
 & $\Obj_0 = \num{2.9e+01} $  & L-BFGS & 781 & 1672 & 1.00e+00 & 2.21e-10 &  \\
 & ${\|\Grad_0\|} = \num{2.3e-02} $  & N-GMRES-B & 2000 & 4440 & 1.64e+00 & 9.19e-03 & $\times$ \\
\addlinespace
\multirow{3}{*}{NCB20} & $ n = \num{5010} $  & O-ACCEL-B & 526 & 1610 & -1.13e+03 & 3.11e-08 &  \\
 & $\Obj_0 = \num{1.0e+04} $  & L-BFGS & 467 & 1089 & -1.15e+03 & 1.29e-06 & $\times$ \\
 & ${\|\Grad_0\|} = \num{4.0e+00} $  & N-GMRES-B & 2000 & 4093 & -1.14e+03 & 6.81e-01 & $\times$ \\
\addlinespace
\multirow{3}{*}{NCB20B} & $ n = \num{5000} $  & O-ACCEL-B & 1178 & 4121 & 7.35e+03 & 3.25e-08 &  \\
 & $\Obj_0 = \num{1.0e+04} $  & L-BFGS & 1467 & 4173 & 7.35e+03 & 2.97e-05 & $\times$ \\
 & ${\|\Grad_0\|} = \num{4.0e+00} $  & N-GMRES-B & 2000 & 3013 & 7.35e+03 & 2.72e-05 & $\times$ \\
\addlinespace
\multirow{3}{*}{NONDQUAR} & $ n = \num{5000} $  & O-ACCEL-B & 363 & 1044 & 1.86e-04 & 1.67e-04 &  \\
 & $\Obj_0 = \num{5.0e+03} $  & L-BFGS & 208 & 436 & 9.30e-05 & 1.81e-04 &  \\
 & ${\|\Grad_0\|} = \num{2.0e+04} $  & N-GMRES-B & 480 & 1394 & 3.77e-04 & 1.71e-04 &  \\
\addlinespace
\multirow{3}{*}{PENALTY3} & $ n = \num{200} $  & O-ACCEL-B & 283 & 1798 & 1.00e-03 & 1.50e-03 &  \\
 & $\Obj_0 = \num{1.6e+09} $  & L-BFGS & 3 & 15 & 1.58e+09 & 1.59e+05 & $\times$ \\
 & ${\|\Grad_0\|} = \num{1.6e+05} $  & N-GMRES-B & 2000 & 2165 & 2.37e+172 & 2.37e+172 & $\times$ \\
\addlinespace
\multirow{3}{*}{POWELLSG} & $ n = \num{5000} $  & O-ACCEL-B & 115 & 460 & 3.31e-06 & 1.22e-06 &  \\
 & $\Obj_0 = \num{2.7e+05} $  & L-BFGS & 20 & 49 & 4.23e-10 & 5.52e-07 &  \\
 & ${\|\Grad_0\|} = \num{3.1e+02} $  & N-GMRES-B & 105 & 308 & 3.16e-09 & 1.41e-08 &  \\
\addlinespace
\multirow{3}{*}{POWER} & $ n = \num{10000} $  & O-ACCEL-B & 186 & 637 & 4.03e+04 & 1.52e+04 &  \\
 & $\Obj_0 = \num{2.5e+15} $  & L-BFGS & 38 & 107 & 4.93e+04 & 1.48e+04 &  \\
 & ${\|\Grad_0\|} = \num{2.0e+12} $  & N-GMRES-B & 607 & 1868 & 7.46e+04 & 1.35e+04 &  \\
\addlinespace
\multirow{3}{*}{SCHMVETT} & $ n = \num{5000} $  & O-ACCEL-B & 74 & 208 & -1.50e+04 & 1.02e-08 &  \\
 & $\Obj_0 = \num{-1.4e+04} $  & L-BFGS & 55 & 133 & -1.50e+04 & 9.19e-09 &  \\
 & ${\|\Grad_0\|} = \num{1.1e+00} $  & N-GMRES-B & 78 & 239 & -1.50e+04 & 9.29e-09 &  \\
\addlinespace
\multirow{3}{*}{SINQUAD} & $ n = \num{5000} $  & O-ACCEL-B & 78 & 297 & -6.75e+06 & 4.94e-05 &  \\
 & $\Obj_0 = \num{6.6e-01} $  & L-BFGS & 12 & 45 & -6.76e+06 & 1.38e-05 &  \\
 & ${\|\Grad_0\|} = \num{5.0e+03} $  & N-GMRES-B & 95 & 483 & -6.75e+06 & 7.10e-06 &  \\
\addlinespace
\multirow{3}{*}{SPARSQUR} & $ n = \num{10000} $  & O-ACCEL-B & 134 & 598 & 1.69e-07 & 4.35e-06 &  \\
 & $\Obj_0 = \num{1.4e+07} $  & L-BFGS & 26 & 91 & 1.36e-06 & 2.25e-04 &  \\
 & ${\|\Grad_0\|} = \num{3.2e+04} $  & N-GMRES-B & 214 & 797 & 1.08e-06 & 1.82e-05 &  \\
\addlinespace
\multirow{3}{*}{TOINTGOR} & $ n = \num{50} $  & O-ACCEL-B & 201 & 538 & 1.37e+03 & 1.34e-06 &  \\
 & $\Obj_0 = \num{5.1e+03} $  & L-BFGS & 126 & 265 & 1.37e+03 & 9.81e-07 &  \\
 & ${\|\Grad_0\|} = \num{1.5e+02} $  & N-GMRES-B & 287 & 795 & 1.37e+03 & 1.29e-06 &  \\
\addlinespace
\multirow{3}{*}{TOINTPSP} & $ n = \num{50} $  & O-ACCEL-B & 277 & 963 & 2.26e+02 & 6.29e-08 &  \\
 & $\Obj_0 = \num{1.8e+03} $  & L-BFGS & 127 & 334 & 2.26e+02 & 2.89e-07 &  \\
 & ${\|\Grad_0\|} = \num{3.0e+01} $  & N-GMRES-B & 423 & 1287 & 2.26e+02 & 2.51e-07 &  \\
\addlinespace
\multirow{3}{*}{VARDIM} & $ n = \num{200} $  & O-ACCEL-B & 8 & 28 & 6.86e-02 & 1.15e+02 &  \\
 & $\Obj_0 = \num{3.3e+16} $  & L-BFGS & 1 & 20 & 3.26e+16 & 1.94e+15 & $\times$ \\
 & ${\|\Grad_0\|} = \num{1.9e+15} $  & N-GMRES-B & 4 & 39 & 5.33e+03 & 4.98e+05 &  \\
\bottomrule
\end{tabular}

  \caption{Results from the CUTEst tests.}\label{tbl:cutest_table2}
\end{table}

\bibliography{references}
\ifnlaa
\bibliographystyle{vancouvernat} 
\else
\bibliographystyle{plainnat}
\fi
\end{document}